\documentclass[12pt]{article}
\usepackage[all]{xy}
\usepackage[affil-it]{authblk}
\usepackage{amsmath} 
\usepackage{amssymb} 
\usepackage{makeidx} 
\usepackage{graphicx} 
\usepackage{enumerate}
\usepackage[utf8]{inputenc}
\usepackage{amsthm}
\usepackage{color}
\usepackage{cancel}

\theoremstyle{plain}
\newtheorem{theorem}{Theorem}[section]
\newtheorem{lemma}[theorem]{Lemma}

\newtheorem{corollary}[theorem]{Corollary}
\newtheorem{proposition}[theorem]{Proposition}

\theoremstyle{definition}

\numberwithin{equation}{section}

\newcommand{\Tr}{{\rm Tr_n}}

\newcommand{\Trm}{{\rm Tr_m}}

\long\def\symbolfootnote[#1]#2{\begingroup%
\def\thefootnote{\fnsymbol{footnote}}\footnote[#1]{#2}\endgroup}

\def\00{{\bf 0}}
\def\11{{\bf 1}}
\def\+{\oplus}

\newcommand{\boxtensor}{{\Box\kern-9.03pt\raise1.42pt\hbox{$\times$}}}

\newcommand{\F}{{\mathbb F}}

\newcommand{\N}{{\mathbb N}}

\newcommand{\V}{{\mathbb V}}

\newcommand{\be}{\begin{eqnarray}}
\newcommand{\ee}{\end{eqnarray}}

\usepackage{color}
\def\meidl#1 {\fbox {\footnote {\ }}\ \footnotetext { From Wilfried: {\color{red}#1}}}
\def\hmeidl#1 {}

\begin{document}

\title{Determining the Walsh spectra of Taniguchi's and related APN-functions}

\author{Nurdag\"{u}l Anbar$^{1}$, Tekg\"{u}l Kalayc\i$^{1}$, Wilfried Meidl$^{2}$
\vspace{0.4cm} \\
\small $^1$Sabanc{\i} University,\\
\small MDBF, Orhanl\i, Tuzla, 34956 \. Istanbul, Turkey\\
\small Email: {\tt nurdagulanbar2@gmail.com}\\
\small Email: {\tt tekgulkalayci@sabanciuniv.edu}\\
\small $^2$Johann Radon Institute for Computational and Applied Mathematics,\\
\small Austrian Academy of Sciences, Altenbergerstrasse 69, 4040-Linz, Austria\\
\small Email: {\tt meidlwilfried@gmail.com}
 }

\date{}

\maketitle

\begin{abstract}
We introduce a method based on Bezout's theorem on intersection points of two projective plane curves, for determining the nonlinearity
of some classes of quadratic functions on $\F_{2^{2m}}$. Among those are the functions of Taniguchi 2019, Carlet 2011, and Zhou and Pott 2013, 
all of which are APN under certain conditions. This approach helps to understand why the majority of the functions in those classes have solely 
bent and semibent components, which in the case of APN functions is called the classical spectrum.
More precisely, we show that all Taniguchi functions have the classical spectrum independent from being APN. We determine the nonlinearity
of all functions belonging to Carlet's class and to the class of Zhou and Pott, which also confirms with comparatively simple proofs earlier results 
on the Walsh spectrum of APN-functions in these classes. Using the Hasse-Weil bound, we show that some simple sufficient conditions for the 
APN-ness of the Zhou-Pott functions, which are given in the original paper, are also necessary. 
\end{abstract}

\textbf {Mathematics Subject Classification(2010):}

\section{Introduction}

For a function $f$ from an $n$-dimensional vector space $\V_n$ over $\F_2$ to $\F_2$, the {\it Walsh transform} $\widehat{f}$ is 
the integer valued function
\[ \widehat{f}(u) = \sum_{x\in\V_n}(-1)^{f(x)+\langle x,u\rangle}, \]
where $\langle\,,\,\rangle$ denotes any (nondegenerate) inner product in $\V_n$. The Walsh spectrum 
$\mathcal{W}_f := \{\widehat{f}(u)\,:\,u\in\V_n\}$ is independent from the inner product used in the Walsh transform.
The Boolean function $f$ is called {\it bent} if for all $u\in\V_n$ we have $|\widehat{f}(u)| = 2^{n/2}$, {\it semibent} if 
$\mathcal{W}_f = \{0,\pm 2^{(n+1)/2}\}$ or $\mathcal{W}_f = \{0,\pm 2^{(n+2)/2}\}$, and more general, {\it $s$-plateaued} if
$\mathcal{W}_f = \{0,\pm 2^{(n+s)/2}\}$ for some integer $s$. Clearly $n+s$ is always even. In particular, bent functions 
only exist if $n$ is even.

Let $F$ be a {\it vectorial function} from $\V_n$ to $\V_n$. Then for a nonzero $c\in \V_n$ the Boolean function $F_c(x) = \langle c,F(x) \rangle$ 
is called a {\it component function} of $F$. The {\it (extended) Walsh spectrum} of $F:\V_n\rightarrow\V_n$ is 
then $\mathcal{W}_F = \cup_{c\in\V_n^*}\mathcal{W}_{F_c}$, the union of the Walsh spectra of the component functions.
The nonlinearity $NL(F)$ of $F$, which plays an important role in applications in cryptography (see \cite{carlet,n1}), is then 
\[ NL(F)  = 2^{n-1}-\frac{1}{2}\max_{W\in\mathcal{W}_F}|W|. \]
A function $F:\V_n\rightarrow\V_n$ is called differentially $k$-uniform if  for all nonzero $a \in \V_n$ and $b\in \V_n$, the equation
\[ D_aF(x) = F(x+a) - F(x) = b \]
has at most $k$ solutions. Having applications in cryptography, differentially $2$-uniform functions $F$, called {\it almost perfect nonlinear (APN)} 
functions are of particular interest, see \cite{BloNy,n2}. In some applications it is required that $F$ is a permutation. Since it seems to be hard to find 
APN-permutations of $\V_n$ when $n$ is even, one often considers also differentially $4$-uniform functions, see e.g. \cite{cdp}.
   
Most known examples and infinite classes of APN-functions $F$ on $\V_n$ are quadratic, i.e., all their component functions have algebraic degree 
(at most) $2$, and hence all component functions are plateaued, see \cite{carlet0}. As is well known, if $n$ is odd, all components of a quadratic
(or plateaued) APN-function must be semibent, i.e., $\mathcal{W}_F\in \{0, \pm 2^{(n+1)/2}\}$. Such functions are called {\it almost bent} functions,
by the Sidelnikov-Chabeaud-Vaudenay bound, \cite{cv}, they are the functions on $\V_n$ with highest nonlinearity.

The situation is different for quadratic APN-functions $F:\V_n\rightarrow\V_n$ when $n$ is even. It is known that $F$ must have at 
least $2(2^n-1)/3$ bent component functions with equality if and only if all other component functions are semibent, i.e.,
$\mathcal{W}_F = \{0,\pm 2^{n/2}, \pm 2^{(n+2)/2}\}$.  As all investigated infinite classes of quadratic APN-functions in even dimension 
have this Walsh spectrum, it is often called the {\it classical spectrum}, see \cite{p} for more details. 

In \cite{ye}, Edel showed that the 13 non-equivalent APN-functions in dimension $6$, which Dillon presented in \cite{Dbanf} (see \cite[p.162]{p})
represent all inequivalent classes of APN-functions in dimension $6$, and he pointed out that among those, there is exactly one class which 
does not have the classical spectrum. (Note that in dimension 6, a quadratic function anyway can only be bent, semibent or $4$-plateaued).
As indicated by Schmidt in \cite{kai-uwe}, there are at least three different spectra for APN-functions in dimension $8$. 
This suggests that there is a larger variety of spectra for APN-functions if $n$ is not small. In fact it is not even known if the APN-property implies 
a high nonlinearity. So far, solely the worst case, the case that one component function is affine, hence $n$-plateaued can be excluded, see the discussion in \cite{C18}. 
 
In this article we introduce a method based on Bezout's theorem on intersection points of two projective plane curves, which is very well suitable to 
show a high nonlinearity for some classes of quadratic functions on $\V_n$ (represented in bivariate form). Among those are some known classes of
APN-functions, and the butterfly construction, which besides from one famous exception (\cite{Dillon}), yields differentially $4$-uniform functions.

The paper is organized as follows. In Section \ref{bezout} we explain the method based on Bezout's theorem. In Section \ref{Tguchi} we show that the recently introduced APN-function in Taniguchi \cite{tani} 
has the classical spectrum, and apply our approach also to the known APN-functions with classical spectrum in Carlet \cite{C11}, and Zhou and Pott \cite{zp}.
For the Zhou-Pott function we will also render the APN condition more precisely using the Hasse-Weil bound. Finally we point out that
our method also works well for the butterfly functions in \cite{cdp}.  




\section{Bezout's theorem and the nonlinearity of quadratic functions}
\label{bezout}

\subsection{Bezout's theorem and common zero sets}

We first recall some basic facts related to plane curves over finite fields. For details, we refer to \cite{HKT}. Let $\mathbb{F}$ 
be a field and $\bar{\mathbb{F}}$ be the algebraic closure of $\mathbb{F}$. An affine curve $\mathcal{X}$ is the zero set of a polynomial 
$f(X,Y)\in \bar{\mathbb{F}}[X,Y] $, i.e.,  
\begin{align*}
\mathcal{X}=\lbrace P=(x,y)\in \bar{\mathbb{F}}\times \bar{\mathbb{F}}\; | \; f(x,y)=0\rbrace \ .
\end{align*}
We say that $f(X,Y)$ is a defining polynomial of $\mathcal{X}$ and the degree of $\mathcal{X}$ is the degree of $f(X,Y)$. A component of 
$\mathcal{X}$ is a curve $\mathcal{Y}$ such that the defining polynomial $g(X,Y)$ of $\mathcal{Y}$ divides $f(X,Y)$. 

\noindent Let  $\mathcal{X}$ be a curve with defining equation $f(X,Y)$ and $\ell$ be a line given by $bX-aY+c$, which is not 
a component of $\mathcal{X}$, and suppose that  $P=(x_0,y_0)\in \mathcal{X}\cap \ell$, i.e., $P$ is an intersection point of $\mathcal{X}$ 
and $\ell$. We can parametrize $\ell$ as follows:
\begin{align*}
x=x_0+at \quad y=y_0+bt \;\; \text{for } t\in \bar{\mathbb{F}} \ .
\end{align*}
As $\ell$ is not a factor of $f(X,Y)$, we have
\begin{align*}
f(X,Y)=f(x_0+at, y_0+bt)= h_mt^m+\cdots +h_d t^d\in \bar{\mathbb{F}}[t] \quad \text{with} \; \; h_m \neq 0 \ .
\end{align*}
Then $m:=m(P, \mathcal{X}\cap \ell)$ is called the intersection multiplicity of $\mathcal{X}$ and $\ell$ at $P$. For $P\in \mathcal{X}$, 
\begin{align*}
m_P(\mathcal{X}):= \min_\ell \lbrace  m(P, \mathcal{X}\cap \ell) \rbrace 
\end{align*}
is called the multiplicity of $\mathcal{X}$ at $P$, where the multiplicity is determined over all lines $\ell$ through $P$ which are not a factor of $f(X,Y)$.
If $m_P(\mathcal{X})=1$, then $P$ is called a non-singular point, otherwise it is called singular. It is a well-known fact that $P=(x_0,y_0)$ is a singular point of $\mathcal{X}$ if and only if 
\begin{align*}
\frac{\partial f(X,Y)}{ \partial X}(x_0,y_0)=\frac{\partial f(X,Y)}{\partial Y}(x_0,y_0)=0 \ ,
\end{align*}
where $\partial f/ \partial X$ and $\partial f/ \partial Y$ are the partial derivatives of $f(X,Y)$ with respect to $X$ and $Y$, respectively.

\noindent Let $\mathcal{X}$ and $\mathcal{Y}$ be two plane curves such that $P\in \mathcal{X}\cap \mathcal{Y}$. Then $\mathcal{X}$ and $\mathcal{Y}$ intersect at $P$ with multiplicity 
\begin{align*}
m(P, \mathcal{X}\cap \mathcal{Y}) \geq m_P(\mathcal{X}) m_P(\mathcal{Y}) \ ,
\end{align*}
and equality holds if and only if they do not have a common tangent line at $P$, see \cite[Theorem 3.7]{HKT}. We remark that all concepts above can similarly be defined for the points of the curves at infinity. Then we have the following well-known result known as Bezout's theorem, see \cite[Theorem 3.13]{HKT}.
\begin{theorem}
Let $\mathcal{X}$ and $\mathcal{Y}$ be two projective plane curves of degree $d_1$ and $d_2$, respectively. If $\mathcal{X}$ and $\mathcal{Y}$ do not have a common component, then
\begin{align*}
\sum_{P\in \mathcal{X}\cap \mathcal{Y}}m(P, \mathcal{X}\cap \mathcal{Y})=d_1d_2 \ .
\end{align*} 
\end{theorem}
 \noindent By Bezout's theorem we conclude that $\mathcal{X}$ and $\mathcal{Y}$ intersect in at most $d_1d_2$ distinct points.

\noindent Field extensions $E_1$, $E_2$ of $F$ are called linearly disjoint extensions of $F$ if $E_1\cap E_2=F$. We first recall a well-known 
fact from Galois theory, for which the proof can be also found in \cite[Lemma 3.1]{bbmm}.
\begin{lemma}\label{lem:disjoint}
Let $E_1$, $E_2$ be two linearly disjoint field extensions of $F$. Then any $F$-linearly independent subset 
$\lbrace v_1, \ldots, v_k \rbrace $ of $E_1$ is also linearly independent over $E_2$.
\end{lemma} 
\noindent Lemma \ref{lem:disjoint} is the main tool to show the following result.
\begin{lemma}\label{lem:zero}
Let $k$ be an integer with $\gcd (k,m)=1$ and let $f$ be a linearized polynomial of the form
\[ f(X) = C_0X + C_1X^{2^k} + C_2X^{2^{2k}} + \cdots + C_dX^{2^{dk}} \in \mathbb{F}_{2^m}[X] \]
of degree $2^{dk}$. Then $f(X)$ has at most $2^d$ zeros in $\mathbb{F}_{2^m}$.
\end{lemma}
\noindent The proof of Lemma \ref{lem:zero} can be found in \cite{tr}. We omit it here as we generalize the result as follows.

\noindent Let $f_1(X,Y), f_2(X,Y)\in \mathbb{F}_{2^m}[X,Y]$ be two linearized polynomials. Then the common zero set 
$\mathcal{Z}_{\mathbb{F}_{2^m}}(f_1,f_2)$ of $f_1, f_2$, which is defined by 
\begin{align*}
\mathcal{Z}_{\mathbb{F}_{2^m}}(f_1,f_2)=\lbrace (x,y) \in \mathbb{F}_{2^m} \times \mathbb{F}_{2^m} \, | \; f_1(x,y)=f_2(x,y)=0\rbrace \ ,
\end{align*}
forms an $\mathbb{F}_2$-vector space. By using the vector space structure and the method in \cite{tr} we have the following proposition.
\begin{proposition}\label{pro:zero}
Let $k$ be an integer with $\gcd (k,m)=1$ and let $f_1(X,Y)$, $f_2(X,Y)$ be linearized polynomials of the form
\begin{align}
\label{CDXY}
C_0X + D_0Y + C_1X^{2^k} + D_1Y^{2^k} + \cdots + C_dX^{2^{dk}} + D_dY^{2^{dk}} \in \mathbb{F}_{2^m}[X,Y]
 \end{align}
of degree $2^{d_1k}$ and $2^{d_2k}$, respectively. If $f_1$ and $f_2$ do not have a common factor, then 
\begin{align}\label{eq:cardinality}
 |\mathcal{Z}_{\mathbb{F}_{2^m}}(f_1,f_2)|\leq 2^{d_1+d_2} \ .
\end{align}
\end{proposition}

\begin{proof}
Note that the assumption $\gcd (k,m)=1$ implies that $\mathbb{F}_{2^m}$ and $\mathbb{F}_{2^k}$ are linearly disjoint over $\mathbb{F}_2$. Let $S=\lbrace v_1, \ldots , v_\ell \rbrace$ be a basis for 
$V_1=\mathcal{Z}_{\mathbb{F}_{2^m}}(f_1,f_2)$ over $\mathbb{F}_2$. We consider the $\mathbb{F}_{2^k}$-vector space $V_2$ generated by $S$. Since $\mathbb{F}_{2^m}$ and 
$\mathbb{F}_{2^k}$ are linearly disjoint over $\mathbb{F}_2$, the set $S$ is linearly independent also over $\mathbb{F}_{2^k}$ by Lemma \ref{lem:disjoint}. Consequently, 
\begin{align}\label{eq:dimension}
\mathrm{dim}_{\mathbb{F}_{2^k}}(V_2)=\mathrm{dim}_{\mathbb{F}_{2}}(V_1) \ .
\end{align}
We observe that for any $c\in \mathbb{F}_{2^k}$,
\begin{align*}
f_1(cx,cy)=cf_1(x,y) \quad \text{and} \quad f_2(cx,cy)=cf_2(x,y) 
\end{align*}
as $f_1, f_2$ are of the form $(\ref{CDXY})$, and hence any element of $V_2$ is a common zero of $f_1$ and $f_2$. Note that $V_2$ is a subset of $\mathbb{F}_{2^{km}}$. As $f_1,f_2$ do not have any common factor,
 by Bezout's theorem, we have
\begin{align}\label{eq:car}
|V_2|\leq \mathrm{deg}(f_1)\; \mathrm{deg}(f_2)= 2^{(d_1+d_2)k}  \ .
\end{align}
By Equations \eqref{eq:dimension} and \eqref{eq:car}, we then have $\mathrm{dim}_{\mathbb{F}_{2}}(V_1)\leq d_1+d_2$, which gives the desired conclusion. 
\end{proof}

\begin{corollary}
\label{cor:zero}
Let $k$ be an integer with $\gcd (k,m)=1$ and let $f_1(X,Y),f_2(X,Y)$ be linearized polynomials of the form $(\ref{CDXY})$, which do not have a common factor. 
If $|\mathcal{Z}_{\mathbb{F}_{2^{km}}}(f_1,f_2)|\leq 2^{kd}$, then $|\mathcal{Z}_{\mathbb{F}_{2^{m}}}(f_1,f_2)|\leq 2^{d}$.
\end{corollary}

\subsection{Determining the nonlinearity of a class of quadratic functions}

Recall that for a Boolean function $g:\V_n\rightarrow \F_2$ an element $u\in \V_n$ is a {\it linear structure} of $g$ if the {\it derivative} $D_ag(X) = g(X+a)+g(X)$ in direction $a$ is constant.
The set $\Lambda_g$ of linear structures of $g$ always is a subspace of $\V_n$, the linear space of $g$. Further recall that every quadratic Boolean function is $s$-plateaued, where 
$s$ is the dimension of its linear space, see for instance \cite{carlet0}.
%

Let $f$ be a quadratic Boolean function given in bivariate trace representation as $f(X,Y) = \Trm(F(X,Y))$ for a polynomial $F \in \F_{2^m}[X,Y]$ of algebraic degree $2$. Note that the directional derivative of $F$ in the direction of $(u,v)\in \F_{2^m}\times \F_{2^m} $ is given by 
\begin{align*}
D_{u,v}f(X,Y) = \Trm(F(X+u,Y+v)+F(X,Y)) \ .
\end{align*}
Then $f$ is $s$-plateaued, where $s$ is the dimension of the linear space of $f$
\[  \Lambda_f = \{(u,v)\in \F_{2^m}\times \F_{2^m} \,:\,D_{u,v}f(X,Y)  = c, \; \; c\in \F_{2} \}. \] 
Since $F$ is quadratic, $D_{u,v}f(X,Y)$ is affine. As we are interested in the values $u$, $v$ for which $D_{u,v}f$ is constant, we can consider the linear part 
$\tilde{D}_{u,v}f(X,Y) = D_{u,v}f(X,Y) + f(u,v)$, which hence can be written in the form $\Trm(L_1(X) ) + \Trm(L_2(Y))$ for some linearized polynomials
\[ L_1(X) = \sum_{i=0}^{\delta_1}a_iX^{2^i}\quad  \text{and} \quad L_2(Y) = \sum_{i=0}^{\delta_2}b_iY^{2^i},  \]
where the coefficients depend on $u$ and $v$.  Using the fact that $\Trm(a x^{2^i}) = \Trm(a^{2^{\delta-i}}x^{2^\delta})$ for all $x\in \mathbb{F}_{2^m}$, we get
\begin{equation}
\label{AandB}
\tilde{D}_{u,v}f(X,Y) = \Trm( A(u,v)X^{2^{\delta_1}} + B(u,v)Y^{2^{\delta_2}})
\end{equation}
for some $A,B\in \F_{2^m}[u,v]$ of algebraic degree $1$.  Clearly,  $\tilde{D}_{u,v}f(X,Y)$ is zero if and only if $A(u,v)  = B(u,v) = 0$, i.e., $\Lambda_f$ is the common zero set
$\mathcal{Z}_{\F_{2^m}}(A,B)$ of $A$ and $B$.

This procedure we described in some more detail for $f$ in bivariate trace representation is of course well known. However, in general it is not easy to determine the dimension $s$ of $\Lambda_f$.
For functions for which $A(X,Y)$ and $B(X,Y)$ obtained as in $(\ref{AandB})$ are of the form $(\ref{CDXY})$, the following corollary, which is an immediate consequence of 
Proposition \ref{pro:zero} and  Corollary \ref{cor:zero}, may help.
\begin{corollary}
\label{BezNL}
Let $f(X,Y) = \Trm(F(X,Y))$ be a quadratic function from $\F_{2^m}\times\F_{2^m}$ to $\F_2$, and suppose that the corresponding linearized polynomials $A(X,Y)$ and $B(X,Y)$ in $(\ref{AandB})$
are of the form $(\ref{CDXY})$ with degrees $kd_1$ and $kd_2$ for some integer $k$ with $\gcd(k,m) = 1$. If $A$ and $B$ do not have a common factor, then $f$ is $s$-plateaued with $s\le d_1+d_2$.
If additionally the curves with defining equations $A$ and $B$ have a common point with intersection multiplicity larger than $1$, then $s < d_1+d_2$.  
\end{corollary}
In the next sections we will apply Corollary \ref{BezNL} to some classes of quadratic functions from $\F_{2^m}\times\F_{2^m}$ to $\F_{2^m}\times\F_{2^m}$, more precisely, to all of their component 
functions simultaneously. Among those are the APN-functions in \cite{C11,zp}, the recently introduced Taniguchi APN-functions, and the differentially $4$-uniform butterfly functions, \cite{cdp}.

\section{The spectrum of Taniguchi's and related APN-functions}
\label{Tguchi}

In \cite{C11}, Carlet showed that for $S,T,U,V \in \F_{2^m}$, $ST\ne 0$, and integers $i,j$ such that $\gcd(m,i-j) = 1$, the function $F:\F_{2^m}\times\F_{2^m}\rightarrow\F_{2^m}\times\F_{2^m}$,
\begin{equation}
\label{C11}
F(X,Y) = (XY,G(X,Y)), \quad G(X,Y) = SX^{2^i+2^j} + UX^{2^i}Y^{2^j} + VX^{2^j}Y^{2^i} + TY^ {2^i+1}
\end{equation}
is an APN-function if and only if $G(X,1)  = SX^{2^i+2^j} + UX^{2^i} + VX^{2^j} + T$ has no root in $\F_{2^m}$. 

In \cite{zp}, Zhou and Pott presented for $\alpha \in \F_{2^m}^*$, $m$ even, $\gcd(m,k) = 1$ and
$j\in \N$ an APN-function $F:\F_{2^m}\times\F_{2^m}\rightarrow\F_{2^m}\times\F_{2^m}$ of a similar form, 
 \begin{equation}
 \label{ZhoPo}
 F(X,Y) = (XY,G(X,Y)), \quad G(X,Y) = X^{2^k+1} + \alpha Y^ {(2^k+1)2^j} \ .
 \end{equation}
As a necessary and sufficient condition for the APN-ness of this function, in \cite{zp} the condition that $\alpha\not\in \{ a^{2^k+1}(t^{2^k}+t)^{1-2^j}\,: a,t\in\F_{2^m}  \}$ is given
and it is pointed out in \cite[Corollary 2]{zp} that this condition is satisfied if $j$ is even, and $\alpha$ is a non-cube.

In \cite{C13}, Carlet gave a general APN-criterion for some classes of functions of the form $(XY,G(X,Y))$, which simultaneously explains the APN-ness of the functions 
$(\ref{C11})$ and $(\ref{ZhoPo})$: Let $F:\F_{2^m}\times\F_{2^m}\rightarrow\F_{2^m}\times\F_{2^m}$
\begin{align}
\label{C13} \nonumber
F(X,Y) & = (XY,G(X,Y)), \quad \mbox{with} \\
G(X,Y) &  = P(X^{2^k+1}) + Q(X^{2^k}Y) + R(XY^{2^k}) + S(Y^ {2^k+1}),
\end{align}
for some homomorphisms $P,Q,R$ and $S$ of $\F_{2^m}$ and an integer $k$ with $\gcd(m,k) = 1$. For every $a,b\in \F_{2^m}$ let $T_{a,b}$ be the linear function given by
\begin{equation}
\label{Tab}
T_{a,b}(Y)  =  P(a^{2^k+1}Y) + Q(a^{2^k}bY) + R(ab^{2^k}Y) + S(b^ {2^k+1}Y) .
\end{equation}
If $m$ is odd, then $F$ is APN if and only if $T_{a,b}$ is a permutation for all $a,b$ (not both $0$). If $m$ is even, then $F$ is APN if and only if for all $a,b$ (not both $0$) we have
$\mathrm{ker} (T_{a,b}) \cap \{ u^{2^k+1}(t^{2^k}+t)\,: u,t\in\F_{2^m}  \} = \{0\}$, where $\mathrm{ker} (T_{a,b})$ is the kernel of $T_{a,b}$. 

Recently in Taniguchi \cite{tani}, another APN-function of similar shape was introduced. For an integer $k$ with $\gcd(m,k) =1$, $\alpha\in \F_{2^m}$, $\beta\in\F_{2^m}^*$,
the function $F(X,Y):\mathbb{F}_{2^m}\times \mathbb{F}_{2^m} \mapsto \mathbb{F}_{2^m}\times \mathbb{F}_{2^m}$ 
\begin{equation}
\label{Tani}
F(X,Y) = (XY,G(X,Y)), \quad G(X,Y)=X^{2^{3k}+2^{2k}}+\alpha X^{2^{2k}}Y^{2^{k}}+\beta Y^{2^{k}+1},
\end{equation}
is APN if and only if $G(X,1)=X^{2^{k}+1}+\alpha X+\beta$ has no root in $\mathbb{F}_{2^m}$. As also pointed out in \cite{tani}, if $\alpha =0$, then the function $(\ref{Tani})$ belongs to 
family $(\ref{ZhoPo})$. Hence in the following subsection we will consider Taniguchi's functions for $\alpha \ne 0$. The functions $(\ref{ZhoPo})$ will be dealt with in the subsection thereafter.

%
\subsection{Taniguchi's APN-function.}
We first remark that Taniguchi's function, which in general is CCZ-inequivalent to the APN-functions $(\ref{C11})$ and $(\ref{ZhoPo})$, see \cite{tani}, is also of the form $(\ref{C13})$ with
$P(X)  = X^{2^{2k}}$, $Q(X) = \alpha X^{2^k}$, $R(X) = 0$ and $S(X) = \beta X$. In fact one can confirm the APN-property for the function $(\ref{Tani})$ with Carlet's criterion as follows:
For Taniguchi's function, $T_{a,b}$ defined as in $(\ref{Tab})$ equals
\[ T_{a,b}(Y) = a^{2^{3k}+2^{2k}}Y^{2^{2k}} + \alpha a^{2^{2k}}b^{2^k}Y^{2^k} + \beta b^{2^k+1}Y. \]
If either $a=0$ or $b=0$ then the only solution for $T_{a,b}(Y) = 0$ is $Y=0$. To determine the kernel of $T_{a,b}(Y)$ when $ab \ne 0$, we substitute $bY$ by $z$ and then divide by $z$, which yields
\[ \frac{a^{2^{3k}+2^{2k}}}{b^{2^{2k}}}z^{2^{2k}-1} + \alpha a^{2^{2k}}z^{2^k-1} + \beta b^{2^k}. \]
Dividing by $b^{2^k}$ and replacing $z^{2^k-1}$ with $Z$  we obtain
\[ \left(\frac{a^{2^{2k}}}{b^{2^k}}\right)^{2^k+1}Z^{2^k+1} + \alpha\frac{a^{2^{2k}}}{b^{2^k}}Z + \beta. \]
Finally with $X = (a^{2^{2k}}/b^{2^k})Z$ we see that the kernel of $T_{a,b}(Y)$ is trivial for all $a,b$ with $ab \ne 0$ if and only if
\[ X^{2^k+1} + \alpha X + \beta = 0 \]
does not have a solution in $\F_{2^m}$. Therefore, with Carlet's criterion the APN-property is confirmed. 
%
%

In the next theorem we present the Walsh spectrum of Taniguchi's function. We remark that the function has the classical spectrum independent from the fact whether the APN-condition
is satisfied or not. 
%
\begin{theorem}
\label{TThm}
Let $F:\F_{2^m}\times\F_{2^m}\rightarrow \F_{2^m}\times\F_{2^m}$ be given as $F(X,Y) = (XY,G(X,Y))$ with 
\[ G(X,Y)=X^{2^{3k}+2^{2k}}+\alpha X^{2^{2k}}Y^{2^{k}}+\beta Y^{2^{k}+1}, \]
where $\gcd(m,k) = 1$, $\alpha,\beta\in\F_{2^m}$, $\alpha\beta\ne 0$. Then the Walsh spectrum of $F$ is $W_F = \{0,\pm 2^{n/2}, \pm 2^{(n+2)/2}\}$. In particular, Taniguchi's APN-function 
has the classical spectrum, i.e., $2(2^{2m}-1)/3$ component functions are bent, the remaining component functions are semibent.  	
\end{theorem}
\begin{proof}
%
We start with some preparations and determine $A$ and $B$ described as in $(\ref{AandB})$.
For $\lambda, \mu \in\F_{2^m}$ the component function  $F_{\lambda,\mu}$ of Taniguchi's function is
\[ F_{\lambda,\mu}(X,Y) = {\rm Tr}_m\left(\lambda XY + \mu(x^{2^{2k}+2^{3k}}+\alpha X^{2^{2k}}Y^{2^k}+\beta Y^{2^k+1})\right). \]
Then for $u,v\in \F_{2^m}$ and $\tilde{D}_{(u,v)}F_{\lambda,\mu}(X,Y) = D_{(u,v)}F_{\lambda,\mu}(X,Y) + F_{\lambda,\mu}(u,v)$ we get
\begin{align*}
\tilde{D}_{(u,v)}F_{\lambda,\mu}(X,Y) & = {\rm Tr}_m\left(\lambda(vX+uY) + \mu(X^{2^{2k}}u^{2^{3k}} + X^{2^{3k}}u^{2^{2k}} \right. )  \\
& \qquad \qquad  \left. +  \mu \alpha (X^{2^{2k}}v^{2^k} + Y^{2^k}u^{2^{2k}}) +  \mu \beta (Y^{2^k}v + Yv^{2^k})\right) \\
& = {\rm Tr}_m \left(X\lambda v+ X^{2^{2k}}(\mu u^{2^{3k}}+\mu \alpha v^{2^k}) + X^{2^{3k}}(\mu u^{2^{2k}})\right)  \\
&\qquad \qquad + {\rm Tr}_m \left(Y(\lambda u + \mu \beta v^{2^k})+ Y^{2^k}( \mu \alpha u^{2^{2k}}+\mu \beta v)\right) \\
& = {\rm Tr}_m(AX^{2^{3k}}) + {\rm Tr}_m(BY^{2^k}),
\end{align*}
where
\begin{align*}
A & = \lambda^{2^{3k}}v^{2^{3k}} + \mu^{2^k}u^{2^{4k}} + \mu^{2^k}\alpha^{2^k}v^{2^{2k}} + \mu u^{2^{2k}}, \\
B & = \lambda^{2^k}u^{2^k} + \mu^{2^k}\beta^{2^k}v^{2^{2k}} + \mu \alpha u^{2^{2k}} + \mu \beta v.
\end{align*}
As ${\rm Tr}_m(AX^{2^{3k}}) + {\rm Tr}_m(BY^{2^k}) = 0$ for all $X,Y \in\F_{2^m}$ if and only if $A=B=0$, putting $u = Y$ and $v = X$,
we have the equations
\begin{align*}
A(X,Y) & = \mu^{2^k}\alpha^{2^k}X^{2^{2k}} + \mu Y^{2^{2k}} + \lambda^{2^{3k}}X^{2^{3k}} + \mu^{2^k}Y^{2^{4k}} = 0, \\
B(X,Y) & = \mu \beta X + \lambda^{2^k}Y^{2^k} + \mu^{2^k}\beta^{2^k}X^{2^{2k}} + \mu \alpha Y^{2^{2k}} = 0,
\end{align*}
or equivalently, 
\begin{align}\label{eq:XY}
A(X,Y)&=   \mu^{2^{-k}}\alpha^{2^{-k}}X+ \mu^{2^{-2k}}Y  + \lambda^{2^{k}}X^{2^{k}}+ \mu^{2^{-k}}Y^{2^{2k}} =0 \\ \nonumber
B(X,Y) & = \mu \beta X + \lambda^{2^k}Y^{2^k} + \mu^{2^k}\beta^{2^k}X^{2^{2k}} + \mu \alpha Y^{2^{2k}} = 0 \ .
\end{align}
Note that $A$ and $B$ are of the form required to apply Corollary \ref{BezNL}. In the first step we will illustrate that $A(X,Y)$ and $B(X,Y)$ do not have  a common component.
Then by Corollary \ref{BezNL} we infer that then $F_{\lambda,\mu}$ is $s$-plateaued with $s$ at most $4$. In the second step we will show that the curves defined by $A$ and $B$
have a common point with intersection multiplicity larger than $1$, which implies $s < 4$. Since $s$ has to be even, we conclude that $F_{\lambda,\mu}$ is bent or semibent,
and the proof is completed. 

First of all we consider the case $\mu =0$. Note that in this case $\lambda \neq 0$. Then Equation \eqref{eq:XY} holds if and only if $\lambda^{2^{k}}X^{2^{k}}=\lambda^{2^{k}}Y^{2^{k}}=0$, 
which holds if and only if $X=Y=0$. Therefore any component function $F_{0,\mu}$ is bent. 
%
%

For $\mu \ne 0$ we consider the curves $\mathcal{X}_1$ and $\mathcal{X}_2$ defined by $A(X,Y) $, $B(X,Y)$ in Equation \eqref{eq:XY}, respectively.
%
We observe that $P_1=(1:0:0)$ and $P_2=((\mu \alpha)^{2^{-2k}}:(\mu \beta)^{2^{-k}}:0)$ are the unique points at infinity of $\mathcal{X}_1$ and $\mathcal{X}_2$, respectively. Since 
$\mu \beta \neq 0$, the points $P_1$ and $P_2$ are distinct.
In particular, this shows that $\mathcal{X}_1$ and $\mathcal{X}_2$ do not have a common component. Consequently by Corollary \ref{BezNL}, the component function $F_{\lambda,\mu}$ is $s$-plateaued 
with $s$ at most $4$.
%
%

It remains to show that $\mathcal{X}_1$ and $\mathcal{X}_2$ have a common point with intersection multiplicity larger than $1$. 
Suppose the opposite, i.e.,
suppose that $\mathcal{X}_1$ and $\mathcal{X}_2$ intersect in exactly $2^{4k}$ distinct points. Note that as they do not have any intersection at infinity, all those intersection points are affine. Now we consider the curves 
$\mathcal{Y}_1$ and $\mathcal{Y}_2$ defined by the equations
\begin{align*}
h_1(X,Y)=A(X,Y)B(X,Y) \quad \text{and} \quad h_2(X,Y)=YA(X,Y)+B(X,Y) \ ,
\end{align*}
which have degree $2^{2k+1}$ and $2^{2k}+1$, respectively. Note that $Y$ is not a factor of $B(X,Y)$, otherwise $(1:0:0)$ would be a point of $\mathcal{X}_2$ at infinity. Hence we conclude that $\mathcal{Y}_1$ and $\mathcal{Y}_2$ 
are curves which do not have a common component. By defining equations of $\mathcal{Y}_1$ and $\mathcal{Y}_2$, any intersection point of $\mathcal{X}_1$ and $\mathcal{X}_2$ is also an intersection point of $\mathcal{Y}_1$ and $\mathcal{Y}_2$. By Bezout's theorem, $\mathcal{Y}_1$ and $\mathcal{Y}_2$ intersect in at most $2^{2k+1}(2^{2k}+1)=2^{4k+1}+2^{2k+1}$ points. Moreover, we have 
\begin{align*}
\frac{\partial h_1(X,Y)}{\partial X}&= \mu^{2^{-k}}\alpha^{2^{-k}}B(X,Y)+\mu \beta A(X,Y)\quad \text{and}\quad \\
 \frac{\partial h_1(X,Y)}{ \partial Y}&= \mu^{2^{-2k}}B(X,Y) \ ,
\end{align*}
i.e., any intersection point $P$ of $\mathcal{X}_1$ and $\mathcal{X}_2$ is a singular point of $\mathcal{Y}_1$. This implies that $m(P, \mathcal{Y}_1 \cap \mathcal{Y}_2)\geq 2$.  In particular, we have 
\begin{align*}
\sum_{P\in \mathcal{X}_1 \cap \mathcal{X}_2} m(P, \mathcal{Y}_1 \cap \mathcal{Y}_2) \geq 2^{4k}2=2^{4k+1} \ .
\end{align*}
Furthermore, $P_1=(1:0:0)$ is a common point of $\mathcal{Y}_1$ and $\mathcal{Y}_2$ with 
$m_{P_1}(\mathcal{Y}_1)= 2^{2k}$ and $m_{P_1}(\mathcal{Y}_2)= 2^{2k}+1$.  Hence $\mathcal{Y}_1$ and $\mathcal{Y}_2$ have intersection multiplicity at $P_1$ at least $2^{2k}(2^{2k}+1)=2^{4k}+2^{2k}$. 
However, then the intersection multiplicity of $\mathcal{Y}_1$ and $\mathcal{Y}_2$ is at least
\begin{align*}
2^{4k+1}+2^{4k}+2^{2k}> 2^{4k+1}+2^{2k+1} \ ,
\end{align*}
which gives a contradiction.
\end{proof}
%
\subsection{Carlet's APN-function and the Zhou-Pott function}

As shown in \cite{tqlt}, Carlet's APN-function $(\ref{C11})$ and all functions $(\ref{ZhoPo})$ of  Zhou and Pott which are APN have  the classical spectrum.
For the latter assertion we refer to our Proposition \ref{ZPAPN} below.
In this subsection we point out that also for these two classes the corresponding functions $A(X,Y)$ and $B(X,Y)$ obtained as above,  are of the form required to apply 
our approach via Bezout's theorem. This leads to a quite simple proof for those functions having the classical spectrum.

We start with the function $(\ref{C11})$ and show the spectrum more general for all variations of the function. As we will see, the function has the classical spectrum for
most of the choices of $S,T,U,V$ in $(\ref{C11})$, independent from the property of being APN.
\begin{corollary}
\label{C11Spec}
For integers $i,j$ such that $\gcd(i-j,m)=1$ and $S,T,U,V \in \F_{2^m}$, $ST\ne 0$, let $F:\F_{2^m}\times\F_{2^m}\rightarrow \F_{2^m}\times\F_{2^m}$ be the function 
$F(X,Y) = (XY,G(X,Y))$ with 
\[ G(X,Y) = SX^{2^i+2^j} + UX^{2^i}Y^{2^j} + VX^{2^j}Y^{2^i} + TY^ {2^i+1}. \]
Then the Walsh spectrum of $F$ is $W_F = \{0,\pm 2^{n/2}, \pm 2^{(n+2)/2}\}$, unless $U = \alpha T$, $V=\alpha^{2^{j-i}}T$ and $ S = \alpha^{2^{j-i}+1}T$ for some nonzero
$\alpha \in \F_{2^m}$. In particular, if $F$ is APN, then $F$ has the classical spectrum. 
If $U = \alpha T$, $V=\alpha^{2^{j-i}}T$ and $S = \alpha^{2^{j-i}+1}T$, then the nonlinearity of $F$ is $NL(F)  = 2^{2m-1}-2^{3m/2}$ if $m$ is even, and $NL(F)  = 2^{2m-1}-2^{(3m-1)/2}$ if $m$ is odd.
\end{corollary}
\begin{proof}
For Carlet's function $(\ref{C11})$, the component function $F_{\lambda,\mu}$ for $\lambda,\mu\in\F_{2^m}$ is given by
\[ F_{\lambda,\mu}(X,Y) = {\rm Tr}_m(\lambda XY + \mu(SX^{2^i+2^j} + UX^{2^i}Y^{2^j} + VX^{2^j}Y^{2^i} + TY^ {2^i+1})). \]
With the analog calculations as above, for $\tilde{D}_{u,v}F_{\lambda,\mu}$ we get 
\begin{align*}
A(X,Y) & = \mu SX + \mu VY + \lambda^{2^j}Y^{2^{j-i}} + \mu^{2^{j-i}}S^{2^{j-i}}X^{2^{2(j-i)}} + \mu^{2^{j-i}}U^{2^{j-i}}Y^{2^{2(j-i)}}, \\
B(X,Y) & = \mu UX + \mu TY + \lambda^{2^j}X^{2^{j-i}} + \mu^{2^{j-i}}V^{2^{j-i}}X^{2^{2(j-i)}} + \mu^{2^{j-i}}T^{2^{j-i}}Y^{2^{2(j-i)}}.
\end{align*}
If $\mu =0$, then $A(X,Y) =B(X,Y)=0 $ if and only if $X=Y=0$, i.e., the corresponding component function is bent. From now on we assume that $\mu\neq 0$. 
We set $k=j-i$ and, for simplicity, we replace $\mu S$, $\mu T$,$\mu V$, $\mu U$ by $\tilde S$, $\tilde T$, $\tilde V$, $\tilde U$, respectively. 
Then we have 
\begin{align}
\label{CAB} \nonumber
A(X,Y) & = \tilde SX + \tilde VY + \lambda^{2^j}Y^{2^{k}} + \tilde S^{2^{k}}X^{2^{2k}} + \tilde U^{2^{k}}Y^{2^{2k}} , \\
B(X,Y) & =  \tilde UX + \tilde TY + \lambda^{2^j}X^{2^{k}} + \tilde V^{2^{k}}X^{2^{2k}} + \tilde T^{2^{k}}Y^{2^{2k}}.
\end{align} 
Let $\mathcal{X}_1$ and $\mathcal{X}_2$ be the curves defined by $A(X,Y)$ and $B(X,Y)$, respectively. The points at infinity are $P_1=(\eta: 1:0)$ for $\mathcal{X}_1$ and $P_2=(1:\zeta:0)$
for  $\mathcal{X}_2$, where 
\begin{align*}
\eta^{2^k}=\frac{ \tilde {U}}{\tilde{S}} \quad \text{and} \quad \zeta^{2^k}=\frac{\tilde V}{\tilde T} .
\end{align*}
We first consider the case that $P_1\ne P_2$. By Bezout's Theorem and Equation \eqref{CAB}, the curves $\mathcal{X}_1$ and $\mathcal{X}_2$ intersect in at most $2^{4k}$ points, which are all affine by our assumption. 
Suppose they intersect in exactly $2^{4k}$ points. Similarly as in the proof of Theorem \ref{TThm} we consider 
\begin{align*}
h_1(X,Y)=A(X,Y)B(X,Y) \quad \text{and} \quad  h_2(X,Y)=(X+\eta Y)A(X,Y)+B(X,Y)  \ .
\end{align*}
Let $\mathcal{Y}_1$ and $\mathcal{Y}_2$ be the curves defined by $h_1(X,Y)$ and $h_2 (X,Y)$, respectively. Note that with $\mathcal{X}_1$ and $\mathcal{X}_2$, also $\mathcal{Y}_1$ and $\mathcal{Y}_2$ do not have 
any common component. Furthermore, any intersection point of $\mathcal{X}_1$ and $\mathcal{X}_2$ is also an intersection point of $\mathcal{Y}_1$ and $\mathcal{Y}_2$ and a singular point of 
$\mathcal{Y}_1$. That is, we have
\begin{align*}
\sum_{P\in \mathcal{X}_1 \cap \mathcal{X}_2} m(P, \mathcal{Y}_1 \cap \mathcal{Y}_2) \geq 2^{4k}2=2^{4k+1} \ .
\end{align*}
Moreover, $P_1=(\eta:1:0)$ is a point of $\mathcal{Y}_1$ and $\mathcal{Y}_2$ of multiplicity $2^{2k}$ and $2^{2k}+1$, respectively, i.e., $\mathcal{Y}_1$ and $\mathcal{Y}_2$ intersect at $P_1$ with  multiplicity
at least $2^{2k}(2^{2k}+1)=2^{4k}+2^{2k}$. However, then the intersection multiplicity of $\mathcal{Y}_1$ and $\mathcal{Y}_2$ is at least
\begin{align*}
2^{4k+1}+2^{4k}+2^{2k}> \mathrm{deg}(\mathcal{Y}_1)\; \mathrm{deg}(\mathcal{Y}_2)=(2^{2k}+2^{2k})(2^{2k}+1)=2^{2k+1}(2^{2k}+1) \ ,
\end{align*}
which contradicts Bezout's Theorem.  Therefore, by Proposition \ref{pro:zero} and Corollary \ref{BezNL}, the 
number of solutions of $A(X,Y)=B(X,Y)=0$ is less than $2^4$, hence $F_{\lambda,\mu}$ is bent or semibent.

To consider the case $P_1 = P_2$ we first note that this implies $\tilde{U}\tilde{V}\ne 0$.
Observe that $P_1=P_2$ if and only if $\tilde{S}/\tilde{V} = \tilde{U}/\tilde{T} = \alpha$ for some nonzero $\alpha\in \F_{2^m}$, or equivalently
\begin{align}\label{eq:SUVT}
\tilde{S} = \alpha\tilde{V} \quad \text{and} \quad \tilde{U} = \alpha\tilde{T} \ .
\end{align}
In this case, by Equation \eqref{eq:SUVT} and by replacing $A(X,Y)$ with $A(X,Y)+\alpha^{2^k}B(X,Y)$ we obtain 
\begin{align*}
A(X,Y)&=\alpha (\alpha^{2^k}\tilde{T}+ \tilde{V})X+(\alpha^{2^k}\tilde{T}+ \tilde{V})Y+\lambda^{2^j}Y^{2^k}+\alpha^{2^k}\lambda^{2^j}X^{2^k}\\
B(X,Y)&=\alpha \tilde{T}X+\tilde{T}Y+\lambda^{2^j}X^{2^k}+\tilde{V}^{2^k}X^{2^{2k}}+\tilde{T}^{2^k}Y^{2^{2k}} \ .
\end{align*}

\noindent We first consider the case $\lambda=0$, for which we have 
\begin{align*}
A(X,Y)&=\alpha (\alpha^{2^k}\tilde{T}+ \tilde{V})X+(\alpha^{2^k}\tilde{T}+ \tilde{V})Y\\
B(X,Y)&=\alpha \tilde{T}X+\tilde{T}Y+\tilde{V}^{2^k}X^{2^{2k}}+\tilde{T}^{2^k}Y^{2^{2k}} \ .
\end{align*}
We have to distinguish two cases.
\begin{itemize}
\item[(i)] $\alpha^{2^k}\tilde{T}+ \tilde{V}=0$.\\
In this case we have 
$A(X,Y)=0$ and
\begin{align*}
B(X,Y)=\tilde{T}(\alpha X+Y)+\tilde{T}^{2^k}(\alpha X+Y)^{2^{2k}} \ .
\end{align*}
Set $Z=\alpha X+Y$, then we have $B(X,Y)=B(Z)=\tilde{T}Z+\tilde{T}^{2^k}Z^{2^{2k}}$. Note that for $z\in \mathbb{F}_{2^m}$ such that $B(z)=0$, we have $2^m$ pairs $(x,y)\in \mathbb{F}_{2^m}\times  \mathbb{F}_{2^m} $ such that $B(x,y)=0$. Since $\gcd (k,m)=1$, we have $\tilde{T}Z+\tilde{T}^{2^k}Z^{2^{2k}}=0$ if and only if $Z^{2^k+1}=\tilde{T}^{-1}$. Hence we conclude that the number of solutions of $B(X,Y)=0$ is $2^{m+2}$ if $m$ is even and $2^{m+1}$ if $m$ is odd.
\item[(ii)] $\alpha^{2^k}\tilde{T}+ \tilde{V} \neq 0$, i.e., $\alpha^{2^k}\neq \tilde{V}/\tilde{T}$.\\
Then we have 
\begin{align*}
A(X,Y)&=\alpha X+Y \\
B(X,Y)&=\tilde{T}(\alpha X+Y)+\tilde{V}^{2^k}X^{2^{2k}}+\tilde{T}^{2^k}Y^{2^{2k}} \ .
\end{align*}
The points at infinity of $\mathcal{X}_1$ and $\mathcal{X}_2$ defined by $A(X,Y)$ and $B(X,Y)$ are $P_1=(1: \alpha:0)$ and $P_2=(1:\eta:0)$, respectively, where $\eta^{2^k}=\tilde{V}/\tilde{T}$. 
Note that $P_1=P_2$ if  and only if $\alpha =\eta$, which holds if and only if $\alpha^{2^k}= \tilde{V}/\tilde{T}$, which is excluded. Hence $\mathcal{X}_1$ and $\mathcal{X}_2$ have distinct points at infinity, therefore they do not have
a common component. By Proposition \ref{pro:zero} the number of solutions of $A(X,Y)=B(X,Y)=0$ is at most $2^2$, thus $F_{\lambda,\mu}$ is bent or semibent. In fact, the intersection multiplicity at $(0,0)$ is greater than $1$ as the curves 
have the same tangent line at $(0,0)$, namely $\alpha X+Y=0$. Consequently, $F_{0,\mu}$ is bent.
\end{itemize}

\noindent Now we consider the case $\lambda \neq 0$. Similarly, we have $P_1=(1: \alpha:0)$ and $P_2=(1:\eta:0)$, where $\eta^{2^k}=\tilde{V}/\tilde{T}$. Then $P_1=P_2$ implies that $\alpha^{2^k}= \tilde{V}/\tilde{T}$, i.e., we have 
\begin{align*}
A(X,Y)&= \lambda^{2^j}(\alpha^{2^k} X^{2^k}+Y^{2^k})= \lambda^{2^j}(\alpha X+ Y)^{2^k} \\
B(X,Y)&=\tilde{T}(\alpha X+Y)+\lambda^{2^j} X^{2^{j}}+\tilde{T}^{2^k}(\alpha X+Y)^{2^{2k}} \ .
\end{align*}
Then $A(X,Y)=B(X,Y)=0$ if and only if $\alpha X+Y= \lambda^{2^j} X^{2^{j}}=0$. This holds if and only if $X=Y=0$, i.e., the corresponding component function is bent. Suppose $P_1\neq P_2$, i.e., $\mathcal{X}_1$ and $\mathcal{X}_2$ 
do not have a common component. Then by Proposition \ref{pro:zero}, and since we have a nontrivial intersection at $(0,0)$ (because $\mathcal{X}_1$ and $\mathcal{X}_2$ have the same tangent line at $(0,0)$, namely $\alpha X+Y$), the number of solutions 
of $A(X,Y)=B(X,Y)=0$ is at most $ 2^2$. Therefore $F_{\lambda,\mu}$ is bent or semibent.

Finally note that if $F$ does not have the classical spectrum, i.e. if $U = \alpha T$, $V=\alpha^{2^k}T$ and $S = \alpha^{2^k+1}T$, with $k = j-i$, then $G(X,1)  = SX^{2^i+2^j} + UX^{2^i} + VX^{2^j} + T$ has the root
$\gamma$ where $\gamma^{2^i}  = \alpha^{-1}$. Hence $F$ is not APN.
\end{proof} 
We now turn our attention to the functions $F$ of the form $(\ref{ZhoPo})$. As pointed out in Corollary 2 in \cite{zp}, if $j$ is even and $\alpha$ is a non-cube,
then $F$ is APN. We need the following proposition to show that otherwise $F$ will never be APN.

\begin{proposition}\label{pro:cube}
Let $m$ be an even integer and $k$ be an integer with $\gcd(m,k) = 1$. Then the curve $\mathcal{X}$ defined by the equation
$X^3-\alpha (T^{2^k}+T) $ always has a solution $(x_0,t_0)\in \F_{2^m} \times \F_{2^m}$ for any nonzero $\alpha \in \F_{2^m} $. 
\end{proposition}
\begin{proof}
We first investigate the case $k < m/2$.  Consider the rational function field $\F_{2^m}(t)$. It is a well-known fact that there is a one-to-one correspondence between the places of $\F_{2^m}(t)$ 
and the irreducible polynomials over $\F_{2^m} $, except for the place at infinity. We consider the extension $F=\F_{2^m}(t,x)$ of $\F_{2^m}(t)$ given by $x^3=\alpha (t^{2^k}+t) $. Note that 
$F/ \F_{2^m}(t)$ is a Kummer extension as $\F_{2^m}$ contains a $3$-rd root of unity, see \cite[Proposition 3.7.3]{sti}. The ramified places of $\F_{2^m}(t)$ are the pole of $t$ and the places 
corresponding the factors of $t^{2^k}+t$, which are totally ramified. Therefore, the degree 
$\mathrm{deg}\left( \mathrm{Diff}\left( F/ \F_{2^m}(t) \right)\right)$ of the Different divisor of $F/ \F_{2^m}(t)$ is $2({2^k}+1)$. Then by Hurwitz genus formula \cite[Theorem 3.4.13]{sti} the genus 
$g(F)$ of $F$ is given by 
\begin{align*}
2g(F)-2=3(-2)+\mathrm{deg}\left( \mathrm{Diff}\left( F/ \F_{2^m}(t) \right)\right)=-6+2({2^k}+1) \ ,
\end{align*}
i.e., $F$ is a function field of genus $g(F)={2^k}-1$. Note that $F$ is a function field with full constant field $\F_{2^m}$ since there is a totally ramified place of $\F_{2^m}(t)$ in the extension $F/ \F_{2^m}(t)$. The Hasse-Well bound \cite[Theorem 5.2.3]{sti} 
then implies that the number $N(F)$ of rational places of $F$ satisfies
\begin{align*}
N(F)\geq 2^m+1-2({2^k}-1)2^{m/2} \ .
\end{align*}
As is well-known, each non-singular point of the curve $\mathcal{X}$ defined by  $X^3=\alpha (T^{2^k}+T) $ corresponds to a unique rational place. Note that $\mathcal{X}$ has no affine singular point 
and there is a unique rational place corresponding to the point at infinity, namely the unique place of $F$ lying over the place of $\F_{2^m}(t)$ at infinity. As a result, the number of affine points 
$N(\mathcal{X})$ of $\mathcal{X}$ satisfies 
\begin{align}\label{eq:number}
N(\mathcal{X})= N(F)-1\geq  2^{m/2} \left( 2^{m/2} -2({2^k}-1)\right) \ .
\end{align}
Then Equation \eqref{eq:number} implies that $N(\mathcal{X})>0$ as $k< m/2$. Hence there exists $(x_0,t_0)$ such that  $x_0^3=\alpha (t_0^{2^k}+t_0)$. Note that $k \neq m/2$ as $\gcd(m,k) = 1$. 
For $k>m/2$, we set $\ell:=m-k <m/2$ and replace $T$ by $T^\ell$, obtaining $X^3=\alpha (T^{2^\ell}+T)$.
\end{proof}
\begin{corollary}\label{cor:S}
Let $m$ be an even integer, and $k$ be an integer with $\gcd(m,k) = 1$. If $j\in \N$ is odd, then $S = \{ a^{2^k+1}(t^{2^k}+t)^{1-2^j}\,: a,t\in\F_{2^m} \}=\F_{2^m}$.
\end{corollary}
\begin{proof}
Let $H=\{x^3\,: x\in\F_{2^m} \setminus \lbrace 0\rbrace \}$ the set of cubes, a proper subgroup of the multiplicative group of $\F_{2^m}$. Since $j$ is odd, hence $\gcd(2^j-1,3) = 1$,
by Proposition \ref{pro:cube}, the set $ \{(t^{2^k}+t)^{1-2^j}\,: t\in\F_{2^m} \}$ contains an element $(t_0^{2^k}+t_0)^{1-2^j}$ from any fixed coset $C$ of the subgroup of cubes $H$. If $a$ 
runs through $\F_{2^m} $, hence $a^{2^k+1}$ runs though $H$,  $ a^{2^k+1}(t_0^{2^k}+t_0)^{1-2^j}$ runs through 
the coset $C$. Therefore $S$ contains all elements of $\F_{2^m}$.
\end{proof}
\begin{proposition}
\label{ZPAPN}
For $\alpha \in \F_{2^m}^*$, $m$ even,  and integers $j,k$ with $\gcd(m,k) = 1$, let $F:\F_{2^m}\times\F_{2^m}\rightarrow\F_{2^m}\times\F_{2^m}$ be defined as
\begin{equation*}
F(X,Y) = (XY,G(X,Y)), \quad G(X,Y) = X^{2^k+1} + \alpha Y^ {(2^k+1)2^j}, 
\end{equation*}
Then $F$ is APN if and only if $j$ is even, and $\alpha$ is a non-cube in $\F_{2^m}$
\end{proposition}
\begin{proof}
As shown in \cite[Theorem 7]{zp}, $F$ is APN if and only if $\alpha$ is not contained in the set $S = \{ a^{2^k+1}(t^{2^k}+t)^{1-2^j}\,: a,t\in\F_{2^m} \}$. Clearly, if $j$ is even, then $S$
contains exactly all the cubes of $\F_{2^m}$ Hence if $j$ is even, then $F$ is APN if and only if $\alpha$ is a non-cube. If $j$ is odd, then $S=\F_{2^m}$ by Corollary \ref{cor:S}. Hence $F$ is not APN. 
\end{proof}

As shown in Theorem 2.1 in \cite{tqlt}, all APN-functions of the form $(\ref{ZhoPo})$ with $j$ even and $\alpha$ a non-cube, therefore by Proposition \ref{ZPAPN} all of the Zhou-Pott APN-functions,
have the classical spectrum. We close this section with a short proof for Theorem 2.1 in \cite{tqlt}. More general, we exactly describe all functions $F$ of the form $(\ref{ZhoPo})$ with classical spectrum,
including the case when $m$ is odd and determine the nonlinearity of the remaining functions. We note that when $m$ is odd then $F$ cannot be APN, which can easily be seen from the original proof in \cite{zp} or with Carlet's criterion. 
\begin{corollary}
For $\alpha \in \F_{2^m}^*$  and integers $j,k$ with $\gcd(m,k) = 1$, let $F:\F_{2^m}\times\F_{2^m}\rightarrow\F_{2^m}\times\F_{2^m}$ be defined as
\begin{equation*}
F(X,Y) = (XY,G(X,Y)), \quad G(X,Y) = X^{2^k+1} + \alpha Y^ {(2^k+1)2^j}.
\end{equation*}
Then $F$ has only bent and semibent components if and only if $m$ is odd, or $m$ is even and $\alpha$ is a non-cube. 
In particular, if $F$ is APN, then $F$ has the classical spectrum.
If $m$ is even and $\alpha$ is a cube, then the nonlinearity of $F$ is $NL(F) =2^{2m-1} - 2^{m+1}$.
\end{corollary}
\begin{proof}
For the Zhou-Pott function, we obtain the equations
\begin{align}
\label{eq:lambda}
A(X,Y) & = \mu^{2^i}Y + \lambda^{2^k}X^{2^k} + \mu^{2^{k+i}}Y^{2^{2k}} , \\ \nonumber
B(X,Y) & = \mu \alpha X + \lambda^{2^k}Y^{2^k} + \mu^{2^k}\alpha^{2^k}X^{2^{2k}}.
\end{align}
If $\mu=0$, then $A(X,Y)=B(X,Y)=0$ if and only if
$X=Y=0$, i.e., the corresponding component function is bent. Now suppose that $\mu \neq 0$. If $\lambda \neq 0$, then the points at infinity of $\mathcal{X}_1$ and $\mathcal{X}_2$, defined by $A(X,Y)$ and 
$B(X,Y)$, respectively, are $P_1=(1:0:0)$ and $P_2=(0:1:0)$. Hence $P_1\neq P_2$, and $\mathcal{X}_1$ and $\mathcal{X}_2$ do not have a common component. Considering the curves $\mathcal{Y}_1$ and 
$\mathcal{Y}_2$ defined by the equations 
\begin{align*}
h_1(X,Y)=A(X,Y)B(X,Y) \quad \text{and} \quad h_2(X,Y)=YA(X,Y) \ ,
\end{align*}
respectively, as in the proof of Theorem \ref{TThm} we can show that the number of intersection points of $\mathcal{X}_1$ and $\mathcal{X}_2$ is less than $2^{4k}$. Hence the number of solutions of 
$A(X,Y)=B(X,Y)=0$ in $\mathbb{F}_{2^m} \times \mathbb{F}_{2^m} $ is less than $2^4$, which gives the desired conclusion.
 
\noindent For $\lambda=0$, Equations $(\ref{eq:lambda})$ reduce to
\begin{align}\label{eq:lambda0}
A(X,Y) & = \mu^{2^i}Y + \mu^{2^{k+i}}Y^{2^{2k}} = 0, \\ \nonumber
B(X,Y) & = \mu \alpha X + \mu^{2^k}\alpha^{2^k}X^{2^{2k}} = 0.
\end{align}
or equivalently
\begin{align}
\label{16}
X = 0\; \mbox{or}\; X^{2^k+1}= (\mu \alpha)^{-1} \quad \text{and} \quad  Y = 0\; \mbox{or}\; Y^{2^k+1}=\mu^{-2^i}.
\end{align} 
If $m$ is odd, hence $(2^k+1,2^m-1) = 1$, then we have $2$ solutions for $X$ and $Y$, hence $4$ solutions $(x,y)\in \mathbb{F}_{2^m} \times \mathbb{F}_{2^m} $, i.e. $F_{0,\mu}$ is semibent.
If $m$ is even,  hence $(2^k+1,2^m-1) = 3$, we have $4$ solutions or one solution for $X$ respectively $Y$, depending on whether $\mu\alpha$ respectively $\mu$ is a cube or not.
If $\alpha$ is a non-cube, then at most one of $\mu\alpha$ and $\mu$ is a cube. Hence we have one or four solution pairs $(x,y)\in \mathbb{F}_{2^m} \times \mathbb{F}_{2^m} $, and  $F_{0,\mu}$ is bent or semibent.
If $\alpha$ is a cube, then we have $16$ solutions $(x,y)\in \mathbb{F}_{2^m} \times \mathbb{F}_{2^m} $ whenever $\mu$ is a cube, hence $F_{0,\mu}$ is $4$-plateaued.
\end{proof} 
%

\vspace{.5cm}

 We close this section pointing out that our approach is also applicable to the butterfly functions investigated in \cite{cdp}. For an odd integer $m$, let $F:\F_{2^m}\times\F_{2^m}\rightarrow \F_{2^m}\times\F_{2^m}$ be defined as
\begin{equation}
\label{butter}
F(X,Y) = (R(X,Y),R(Y,X))\;\mbox{with}\; R(X,Y) = (X+\alpha Y)^3 + \beta Y^3,\;\alpha,\beta\in\F_{2^m}^*.
\end{equation}
This quadratic function $F$ belongs to the {\it closed butterfly} class. Such functions are CCZ-equivalent to permutations called {\it open butterfly}. Most notably, for $m=3$, $\Tr(\alpha) = 0$, $\alpha \ne 0$ and 
$\beta\in\{\alpha^3+\alpha,\alpha^3+\alpha^{-1}\}$ the function $(\ref{butter})$ is CCZ-equivalent to the only known APN-permutation in an even number of variables in \cite{Dillon}. We refer to \cite{cdp} for details, where 
the functions $(\ref{butter})$ were thoroughly investigated. Amongst others it is shown that $F$ is differentially $4$-uniform if and only if $\beta \ne (1+\alpha)^3$, but the above mentioned case is the only 
one for which $F$ is APN. For the Walsh spectrum for functions $F$ defined as in $(\ref{butter})$ the following Theorem is shown.
\begin{theorem}\cite[Theorem 14]{cdp}
	\label{Thm14}
	If $\beta \ne (1+\alpha)^3$, then the Walsh spectrum of $F$ is $W_F = \{0,\pm 2^m,\pm 2^{m+1}\}$, i.e., (since $F$ is quadratic) all components are bent or semibent.
	If $\beta = (1+\alpha)^3$, then the nonlinearity of $F$ is $2^{2m-1} - 2^{(3m-1)/2}$.
\end{theorem}
%
\noindent Applying our Corollary \ref{BezNL} based on Bezout's theorem, we can shorten the proof of Theorem \ref{Thm14}
to some extent.

Straightforwardly we get our conditions $A(X,Y) = B(X,Y) = 0$ for the component function $F_{\lambda , \mu}$ of $F$ given as in $(\ref{butter})$ as
\begin{align*}
A(X,Y) & = C_1X + C_2Y + C_1^2 X^4 + C_3^2Y^4 , \\
B(X,Y) & = C_3X + C_4Y + C_2^2 X^4 + C_4^2Y^4 ,
\end{align*}
where $C_1 = \lambda + \mu d$, $C_2 =  \lambda\alpha + \mu\alpha^2$, $C_3 = \lambda\alpha^2 + \mu\alpha$, $C_4 = \lambda d + \mu$ with $d=\alpha^3+\beta$,
see Equation (1) in \cite{cdp}. 

Observe that in this case if $A$ and $B$ do not have a common component then by Corollary \ref{BezNL}, $F$ is $s$-plateaued with $s \le 2$, i.e. bent or semibent.

First note that we can assume that $C_1, C_2, C_3, C_4$ are all nonzero: If $C_1 = 0$ and $(C_2, C_3)\ne (0,0)$, then we obtain from $A(X,Y)$ at most $2$ solutions for $Y$ and  from $B(X,Y)$ we obtain at most $2$ solutions for $X$ each solution $Y$. That is, we have at most $4$ solutions. The same arguments applies if  $C_4 = 0$. If $C_1=C_2=C_3=0$, then we obtain $\alpha = 1$ and $\mu = \alpha$ from $C_2=C_3 = 0$ (assuming $\alpha\ne 0$),
and then $b=0$ from $C_1=0$, which we exclude (and implies that also $C_4 = 0$). The same argument applies
if $C_4=C_2=C_3=0$. Hence we can assume that $C_1C_4\ne 0$. If now $C_2C_3 = 0$, then the points at infinity of the curves defined by $A$ and $B$ are obviously different and 
$A$ and $B$ do not have a common component. 

For $C_1C_2C_3C_4 \ne 0$, the points at infinity for $A$ and $B$ are $P_1 = (\zeta,1,0)$ and $P_2 = (\eta,1,0)$, respectively, where $\zeta^2 = C_1/C_3$ and 
$\eta^2 = C_2/C_4$. Hence if $C_1C_4 \ne C_2C_3$, then $A$ and $B$ do not have a common component, and it remains to investigate the case that  $C_1C_4 = C_2C_3$.

Consider the equivalent system $A(X,Y) = 0$, $\tilde{B}(X,Y) = 0$ with
\begin{align*} 
\tilde{B}(X,Y) & = C_2^2A + C_1^2B = (C_2^2C_1+C_1^2C_3)X + (C_2^3+C_1^2C_4)Y \\
& = C_1(C_2^2+C_1C_3)X + C_2(C_2^2+C_1C_3)Y.
\end{align*}
If $C_2^2+C_1C_3 \ne 0$, then we can instead consider the equations 
\[ \bar{B}(X,Y) = C_1X + C_2Y, \quad A(X,X) + \bar{B}(X,Y) = (C_1^ {-2}X + C_3^{-2}Y)^4. \]
Since $C_1C_3^{-2} + C_2C_1^ {-2} = 0$ implies $C_1(C_1C_3+C_2^2) = 0$, a contradiction, we have the unique solution $X=Y=0$, hence $F_{\lambda,\mu}$ is bent.  

It remains to investigate the case $C_1C_4 = C_2C_3$ and $C_2^2+C_1C_3 = 0$ (for which $B$ is a constant multiple of $A$). We reproduce here the relevant part of the proof of 
Theorem 14 in \cite{cdp} as follows. We have
\begin{align*}
C_2^2 + C_1C_3 & = \lambda^2\alpha^2 + \mu^2\alpha^4 + (\lambda + \mu(\alpha^3+\beta))(\lambda \alpha^2 + \mu \alpha) \\
& = \alpha\mu(\mu \beta + \lambda(\alpha^4+\alpha\beta+1)).
\end{align*}
Hence, $C_2^2 + C_1C_3 = 0$ if and only if
\begin{equation}
\label{condition}
\frac{\mu}{\lambda} = \beta^{-1}(\alpha^4+\alpha\beta+1).
\end{equation}
Then,
\begin{align*}
C_2C_3+C_1C_4 & = (\lambda \alpha+\mu\alpha^2)(\lambda \alpha^2+\lambda \alpha)+(\lambda+\mu(\alpha^3+\beta))(\lambda(\alpha^3+\beta)+\mu) \\
& = \mu^2(\alpha^3+\beta) + \lambda\mu((\alpha+1)^3+\beta)^2 + \lambda^2\beta.
\end{align*}
Using the condition in $(\ref{condition})$ we get
\begin{align*}
C_2C_3+C_1C_4 & = \frac{\lambda}{\beta}((\alpha^4+\alpha\beta+1)^2 + (\alpha^4+\alpha\beta+1)((\alpha+1)^3+\beta)^2+\beta^2) \\
& = \alpha((\alpha+1)^3+\beta)^2((\alpha^2+1)\alpha+\beta).
\end{align*}
If $(\alpha^2+1)\alpha+\beta = 0$, then with $(\ref{condition})$ we have $\lambda = \mu\alpha$, hence $C_3 = 0$, a contradiction.
Consequently, $C_2^2 + C_1C_3 = C_2C_3+C_1C_4 = 0$ if and only if $\beta = (\alpha+1)^3$.

Observe that then $C_1=C_2=C_3=C_4$, and we have the $(m+1)$-dimensional solution space $\Lambda_F = \{(X,Y)\;:\;Y=X\;\mbox{or}\;Y=X+1\}$.

\section*{Acknowledgement}

N.A is supported by B.A.CF-19-01967;  T.K is  supported by TÜBİTAK project 215E200; and W.M. is supported by the FWF Project P 30966.


\begin{thebibliography}{99}

\bibitem{BloNy} C. Blondeau, K. Nyberg, Perfect nonlinear functions and cryptography. Finite Fields Appl. 32 (2015), 120--147.

\bibitem{bbmm} C.Bracken, E. Byrne, N. Markin, G. McGuire, Fourier spectra of binomial APN functions,  SIAM J. Discrete Math. 23 (2009), no. 2, 596--608.

\bibitem{Dillon} K.A. Browning, J.F. Dillon, M.T. McQuistan, A.J. Wolfe,  An APN permutation in dimension six. In: Finite fields: theory and applications,
Contemp. Math., 518, pp. 33--42, Amer. Math. Soc., Providence, RI, 2010. 

\bibitem{cdp} A. Canteaut, S. Duval, L. Perrin, A generalisation of Dillon's APN permutation with the best known differential and nonlinear properties for all 
fields of size $2^{4k+2}$. IEEE Trans. Inform. Theory 63 (2017), 7575--7591.

\bibitem{carlet0} C. Carlet, Boolean Functions for Cryptography and Error Correcting Codes. In: Crama Y., Hammer P. (eds.), Chapter of the monography, Boolean 
Models and Methods in Mathematics, Computer Science, and Engineering, pp. 257--397. Cambridge University Press, Cambridge (2010). 

\bibitem{carlet} C. Carlet, Vectorial Boolean Functions for Cryptography. In: Crama Y., Hammer P. (eds.), Chapter of the monography, Boolean Models and 
Methods in Mathematics, Computer Science, and Engineering, pp. 398--469. Cambridge University Press, Cambridge (2010). 

\bibitem{C11} C. Carlet, Relating three nonlinearity parameters of vectorial functions and building APN functions from bent functions. Des. Codes Cryptogr. 59 (2011),  89--109. 

\bibitem{C13} C. Carlet, More constructions of APN and differentially 4-uniform functions by concatenation. Sci. China Math. 56 (2013), 1373--1384. 


\bibitem{C18} C. Carlet, Characterizations of the differential uniformity of vectorial functions by the Walsh transform. IEEE Trans. Inform. Theory 64 (2018), 6443--6453. 

\bibitem{cv} F. Chabaud, S. Vaudenay, Links between differential and linear cryptanalysis. In: Advances in cryptology--EUROCRYPT '94 (Perugia), 
Lecture Notes in Comput. Sci., 950, pp. 356--365, Springer, Berlin, 1995.

\bibitem{Dbanf} J.F. Dillon, Slides from talk given at "Polynomials over Finite Fields and Applications" given at Banff International Research Station (2006).

\bibitem{ye} Y. Edel, Quadratic APN functions as subspaces of alternating bilinear forms, In: Proceedings of the Contact Forum Coding Theory and Cryptography III 
at The Royal Flemish Academy of Belgium for Science and the Arts 2009, pp. 11--24, 2011.

\bibitem{HKT} J.W.P. Hirschfeld,  G. Korchmáros, F. Torres, {\it Algebraic curves over a finite field}, Princeton University Press, 2013.

\bibitem{n1} K. Nyberg, Perfect nonlinear S-boxes, In: Advances in cryptology--EUROCRYPT '91 (Brighton, 1991), 
Lecture Notes in Comput. Sci., 547, pp. 378--386, Springer, Berlin, 1991.

\bibitem{n2} K. Nyberg, Differentially uniform mappings for cryptography. In: Advances in cryptology—EUROCRYPT '93 (Lofthus, 1993), Lecture Notes in 
Comput. Sci., 765, pp. 55--64, Springer, Berlin, 1994.

\bibitem{p} A. Pott, Almost perfect and planar functions. Des. Codes Cryptogr. 78 (2016), 141--195.

\bibitem{kai-uwe} K.-U. Schmidt, Slides from talk given at "Fifth Irsee Conference on Finite Geometry" given in Irsee (2017); http://cage.ugent.be/~ml/irsee5/slides/Schmidt.pdf.

\bibitem{sti} H. Stichtenoth, {Algebraic function fields and codes}, $2^{\rm nd}$ Edition, Graduate Texts in Mathematics 254, Springer Verlag, 2009.

\bibitem{tqlt} Y. Tan, L. Qu, S. Ling, C.H. Tan, On the Fourier spectra of new APN functions. SIAM J. Discrete Math. 27 (2013), 791--801.

\bibitem{tani} H. Taniguchi, On some quadratic APN functions, Des. Codes Cryptogr., to appear.

\bibitem{tr} H.  M. Trachtenberg,  On  the  Cross-Correlation  Functions  of  Maximal  Linear  Sequences, Ph.D. dissertation, University of Southern California, Los Angeles, 1970.


\bibitem{zp} Y. Zhou, A.  Pott, A new family of semifields with 2 parameters. Adv. Math. 234 (2013), 43--60.

\end{thebibliography}
\end{document}